\documentclass[11pt]{amsart}
\usepackage{tikz}
\definecolor{onefill}{RGB}{210,228,250}
\definecolor{twofill}{RGB}{247,236,214}
\usepackage{float}
\usepackage[utf8]{inputenc}
\usepackage[T1]{fontenc}
\usepackage{amsmath,amssymb,amsthm,mathtools}
\usepackage[pagebackref]{hyperref}
\usepackage{microtype}

\usepackage{geometry}
\title[Modularity of point counts on plane curves]{Modularity of Point
Counts for the Curves \(X^a=Y^b\): New Rogers--Ramanujan
Identities}

\author{Kenny Lau and Ken Ono}
\address{Axiom Math, 124 University Avenue, Palo Alto, CA 94301}
\email{kenny@axiommath.ai}
\email{ken@axiommath.ai}
\date{}

\newtheorem{theorem}{Theorem}[section]
\newtheorem{lemma}[theorem]{Lemma}
\newtheorem{proposition}[theorem]{Proposition}
\newtheorem{corollary}[theorem]{Corollary}
\newtheorem*{conjecture}{Conjecture}
\theoremstyle{remark}
\newtheorem{remark}[theorem]{Remark}
\theoremstyle{definition}

\newtheorem*{example}{Example}

\newcommand{\Z}{\mathbb Z}
\newcommand{\R}{\mathbb R}
\newcommand{\Fq}{\mathbb F_q}

\newcommand{\defn}{\textbf}
\newcommand{\qbinom}[2]{\genfrac{[}{]}{0pt}{}{#1}{#2}_{\!q}}
\newcommand{\qbinomQ}[3]{\genfrac{[}{]}{0pt}{}{#1}{#2}_{\!#3}}
\newcommand{\abs}[1]{\left|#1\right|}
\newcommand{\norm}[1]{\left\lVert#1\right\rVert}
\newcommand{\set}[1]{\left\{#1\right\}}
\newcommand{\ind}{\mathbf 1}
\newcommand{\Gap}{\mathcal G}
\DeclareMathOperator{\inv}{inv}
\DeclareMathOperator{\coinv}{coinv}
\DeclareMathOperator{\dist}{dist}
\DeclareMathOperator{\Nilp}{Nilp}

\begin{document}

\begin{abstract}
For coprime \(1<a<b\), let \(M_n^{a,b}(\Fq)\) be the set of commuting
pairs of nilpotent \(n\times n\) matrices over \(\Fq\) with
\(X^a=Y^b\).  Huang, Jiang, and Oblomkov assembled their orders as an Eulerian $q$-series
 \(Z_{a,b}(q)\). They
conjectured that it is an explicit product \(P_{a,b}(q)\) involving Jacobi's
theta function and Dedekind's eta-function, implying the threefold equality
\[
\underbrace{\vphantom{\Bigg|}
\prod_{n\geq1}(1-q^n)\cdot
\Biggl(\,\sum_{n=0}^{\infty}
\frac{\abs{M_n^{a,b}(\Fq)}}
{\abs{\operatorname{GL}_n(\Fq)}}
\Biggr)\Bigg|_{q\mapsto q^{-1}}}_{\text{point count}}
\;=\;
\underbrace{\vphantom{\Bigg|}Z_{a,b}(q)}_{\text{\(q\)-series}}
\;=\;
\underbrace{\vphantom{\Bigg|}P_{a,b}(q)}_{\text{theta quotient}}.
\]
If true, the point count on \(X^a=Y^b\) is essentially a modular
function on \(\Gamma(a+b)\).  The conjecture is layered in \(a\), with an identity for each \(b\).  The
\(a=2\) layer is classical, including identities of Rogers--Ramanujan and
Andrews--Gordon. 
For \(a\geq3\), nothing was known.  We prove the \(a=3\) layer in
full: a new infinite family of Rogers--Ramanujan identities, and a
geometric origin for Warnaar's products.
AxiomProver verified these new identities in Lean assuming existing
literature.
\end{abstract}



\maketitle
\section{Introduction and statement of results}

 The Rogers--Ramanujan (RR) identities are among the most celebrated in
mathematics:
\begin{equation}\label{RRIdentities}
\sum_{n=0}^\infty\frac{q^{n^2}}{(q)_n}
=\prod_{n=0}^\infty\frac{1}{(1-q^{5n+1})(1-q^{5n+4})}
\ \ \ \  {\text {\rm and}} \ \ \  
\sum_{n=0}^\infty\frac{q^{n^2+n}}{(q)_n}
=\prod_{n=0}^\infty\frac{1}{(1-q^{5n+2})(1-q^{5n+3})},
\end{equation}
where \((q)_m:=\prod_{j=1}^m(1-q^j)\) is the \defn{\(q\)-Pochhammer
symbol}.  Proved by Rogers in 1894 \cite{Rogers1894} and
rediscovered by Ramanujan and by Schur \cite{Schur1917}, each equates
an Eulerian \(q\)-series with an infinite product, and each counts
partitions: those with parts \(\equiv\pm1\) (resp.\ \(\pm2\)) modulo
\(5\) are equinumerous with those whose parts differ by at least \(2\)
(resp.\ and omit \(1\)).  Beyond partitions, \eqref{RRIdentities}
recur across mathematics and physics (for example, see
\cite{Sills}).

In 1974, Andrews \cite{Andrews74} found a significant
generalization of \eqref{RRIdentities}.  The \defn{Andrews--Gordon (AG) identities} form an
infinite family, one for each \(k\geq1\) and each residue
\(1\leq i\leq k+1\):
\begin{equation}\label{eq:AG}
\sum_{r_1\geq\cdots\geq r_k\geq0}
\frac{q^{r_1^2+\cdots+r_k^2+r_i+\cdots+r_k}}
{(q)_{r_1-r_2}\cdots(q)_{r_{k-1}-r_k}(q)_{r_k}}
=\prod_{\substack{n\geq1\\ n\not\equiv0,\pm i\!\!\pmod{2k+3}}}
\frac{1}{1-q^n}.
\end{equation}
The \(k=1\) case recovers
\eqref{RRIdentities}.  These identities are the analytic counterpart
of  work of Gordon \cite{Gordon61}, and, like RR, they
arise as specialized characters of the affine Kac--Moody
algebra \(A_1^{(1)}\). 
How far the RR/AG phenomenon extends, and what produces it, is a
recurring question.

One answer is due to Griffin, the second author, and Warnaar (GOW)
\cite{GOW}, who embedded \eqref{RRIdentities}--\eqref{eq:AG} into
\emph{doubly-infinite} {\bf GOW} families whose sum sides are
Hall--Littlewood polynomials \cite{Macdonald} at geometric progressions
and whose product sides are characters of the affine Lie algebras
\(A_{2n}^{(2)}\), \(C_n^{(1)}\), \(D_{n+1}^{(2)}\), and
\(A_{n-1}^{(1)}\).  The framework has since been extended by Bartlett
and Warnaar \cite{BartlettWarnaar}, by Rains and Warnaar
\cite{RainsWarnaar} through bounded Littlewood identities for
Macdonald--Koornwinder polynomials, and by Kanade, Russell, Tsuchioka,
and Warnaar \cite{KRTW}, who gave a Lie-algebraic interpretation of
conjectures of Capparelli, Meurman, Primc, and Primc.

The \(A_2\) case has its own history.  Andrews, Schilling, and Warnaar
\cite{ASW} proved an \(A_2\) Bailey lemma and derived the \(A_2\)
Rogers--Ramanujan identities, reproved by Warnaar \cite{WarnaarHL} via
Hall--Littlewood functions.  Cylindric partitions then brought a wave
of activity \cite{CorteelWelsh, CorteelDousseUncu, KanadeRussell}, out
of which Warnaar obtained the \(A_2\) Andrews--Gordon identities
\cite{WarnaarAG}.

What unites this work, from Rogers and Andrews through GOW to the
cylindric-partition program, is that the identities arise
from representation theory and symmetric functions: affine Lie algebra
characters, Hall--Littlewood polynomials, Bailey pairs, Virasoro
modules.  A geometric route of a different kind was opened by Bruschek,
Mourtada, and Schepers \cite{BMS}, who realized the first
Rogers--Ramanujan identity in the Hilbert--Poincar\'e series of the
\emph{arc space} of the double point \(X^2=0\). The viewpoint has since
become a productive program \cite{Afsharijoo, ADJM}.  The mechanism we
take up is geometric in another sense entirely: the identities come
from point counts over finite fields.

For \(a,b>1\) with \(\gcd(a,b)=1\), let
\(\Gamma_{a,b}=\langle a,b\rangle\subset\mathbb N=\{0,1,\ldots\}\) be
the semigroup generated by \(a\) and \(b\), and let
\(\Gap_{a,b}:=\mathbb N\setminus\Gamma_{a,b}\) be its finite \defn{gap
set}.  By Sylvester's theorem the largest gap, the \defn{Frobenius
number}, is \(f_{a,b}:=\max\Gap_{a,b}=ab-a-b\) (see, e.g.,
\cite[Chapter~2]{RamirezAlfonsin}).  With \(a,b\) fixed we write
\(\Gamma\) and \(\Gap\).
Order \(\Gap\) by divisibility in the semigroup: \(i\preceq j\) when
\(j-i\in\Gamma\).  The points
\(\mathcal Y:=\set{(x,y):x,y\in\Z_{\geq1},\ ab-ax-by>0}\) sit below the
anti-diagonal of an \(a\times b\) rectangle, and \(g(x,y):=ab-ax-by\)
carries \(\mathcal Y\) bijectively onto \(\Gap\).  Writing the inverse as
\(i\mapsto(x(i),y(i))\), the semigroup order becomes coordinatewise:
\[
i\preceq j
\quad\Longleftrightarrow\quad
x(i)\geq x(j)
\qquad\text{and}\qquad
y(i)\geq y(j).
\]
So \(\Gap\) is literally a Young diagram in the \(a\times b\)
rectangle, its upward-closed subsets are subdiagrams encoded by
rational Dyck paths, and its maximum is the corner cell \((1,1)\), of
value \(f_{a,b}\).  This picture links the semigroup to \(q,t\)-Catalan combinatorics,
Hilbert schemes on \(X^a=Y^b\), and HOMFLY--PT polynomials of
\((a,b)\)-torus knots \cite{GorskyMazin,ORS2018}.

\begin{example}\label{ex:34-gaps}
Take \((a,b)=(3,4)\): then \(f_{3,4}=5\) and
\(\Gap_{3,4}=\{1,2,5\}\).  The map \(g^{-1}\) sends \(5,1,2\) to
\((1,1),(1,2),(2,1)\), so \(5\) is the maximum and \(1,2\) are
incomparable (Figure~\ref{fig:gap34}).
\begin{figure}[H]
\centering
\begin{tikzpicture}[scale=0.7, line join=round]
  \begin{scope}[gray!45]
    \foreach \x in {0,...,4} \draw[very thin] (\x,0) -- (\x,3);
    \foreach \y in {0,...,3} \draw[very thin] (0,\y) -- (4,\y);
  \end{scope}
  \draw[gray!55, dashed, very thin] (4,0) -- (0,3);  % ab-ax-by=0
  \definecolor{cellfill}{RGB}{225,235,250}
  \foreach \x/\y/\v in {0/0/5, 0/1/1, 1/0/2} {
    \fill[cellfill] (\x,\y) rectangle ++(1,1);
    \draw[thick] (\x,\y) rectangle ++(1,1);
    \node at (\x+0.5,\y+0.5) {$\v$};
  }
  \node[below left, font=\footnotesize] at (0,0) {$(1,1)$};
  \node[font=\footnotesize, gray!70] at (2.9,2.55) {$3\times4$};
\end{tikzpicture}
\caption{The gap set $\Gap_{3,4}=\{1,2,5\}$ in the $3\times4$
rectangle, each cell labeled by $g(x,y)=12-3x-4y$.}
\label{fig:gap34}
\end{figure}
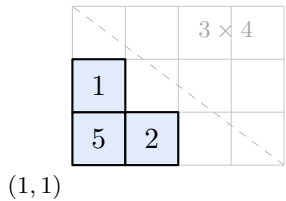
\end{example}

Huang, Jiang, and Oblomkov \cite{HJO} proposed a new source of
Rogers--Ramanujan-type identities, rooted in algebraic geometry rather
than representation theory.  From the enumerative geometry of the plane
curve \(X^a=Y^b\) they produce a \(q\)-series \(Z_{a,b}(q)\),
conjecturally an explicit and essentially modular infinite product with character.  If
true, the RR/AG phenomenon has a geometric origin: the
affine-Lie-algebra identities above are governed by point counts on
these singular curves.

For each \(n\geq1\), they defined the rank \(n\) \(a\times b\) rational
\(q\)-Catalan number, a polynomial with nonnegative integer
coefficients reducing at \(n=1\) to the usual one.  Our interest is its
\(n\to\infty\) limit.  Write \(\ind_E\) for the indicator of a
condition or set \(E\), and put
\begin{displaymath}
U_{a,b}(x):=\ind_{x\geq0}-\ind_{x\geq a}
-\ind_{x\geq b}+\ind_{x\geq a+b}.
\end{displaymath}
Let \(\Z^{\Gap}\) be the integer-valued functions
\(\boldsymbol n=(n_i)_{i\in\Gap}\), extended to \(\Z\) by \(n_i=0\) off
\(\Gap\), and define
\begin{equation}\label{def:quadratic-form}
Q_{a,b}(\boldsymbol n)
  :=\sum_{i,j\in\Gap}U_{a,b}(j-i)n_i n_j.
\end{equation}
The relevant cone is
\begin{displaymath}
\mathcal C_{a,b}:=\set{\boldsymbol n\in\R^{\Gap}:
n_i\geq0,\ n_i\geq n_j\ \text{whenever }i-j\in\Gamma}.
\end{displaymath}
All Pochhammer subscripts below are nonnegative on
\(\mathcal C_{a,b}\), so we may define
\begin{equation}\label{eq:inf-sum-def}
Z_{a,b}(q):=
\sum_{\boldsymbol n\in\Z^{\Gap}\cap\mathcal C_{a,b}}
q^{Q_{a,b}(\boldsymbol n)}
\prod_{i\in\Gap}
\frac{(q)_{n_i-n_{i-a-b}}}
{(q)_{n_i-n_{i-a}}(q)_{n_i-n_{i-b}}},
\end{equation}
with the extension convention supplying the subscripts outside
\(\Gap\).  This agrees with the customary sum over \(\Z^{\Gap}\) under
the convention \(1/(q)_m=0\) for \(m<0\), but avoids multiplying a
vanishing reciprocal by an undefined numerator.
Huang \cite[Theorem~1.3]{Huang2026} proved that \(Q_{a,b}\) satisfies
\[
Q_{a,b}(\boldsymbol n)\geq
\frac{1}{\abs{\Gap}}\norm{\boldsymbol n}_\infty^2
\geq\frac{1}{\abs{\Gap}^2}\norm{\boldsymbol n}_2^2
\qquad(\boldsymbol n\in\mathcal C_{a,b}).
\]
Therefore, \(Z_{a,b}(q)\) is a well-defined formal power series, absolutely
convergent for \(\abs q<1\).

Geometry enters as follows.  Let \(\Nilp_n(\Fq)\) be the nilpotent
\(n\times n\) matrices over \(\Fq\), with \(\Nilp_0(\Fq)\) and
\(\operatorname{GL}_0(\Fq)\) of order \(1\).  For a prime power \(q\), we let
\begin{displaymath}
M_n^{a,b}(\Fq):=\set{(X,Y):X,Y\in\Nilp_n(\Fq),\
XY=YX,\ X^a=Y^b}
\end{displaymath}
and
\[
S_q=\sum_{n=0}^{\infty}
\frac{\abs{M_n^{a,b}(\Fq)}}{\abs{\operatorname{GL}_n(\Fq)}}.
\]
Huang, Jiang, and Oblomkov proved \cite{HJO} that
\begin{equation}\label{eqn:HJO}
S_q=Z_{a,b}(q^{-1})\prod_{n=1}^{\infty}(1-q^{-n})^{-1}<\infty,
\end{equation}
which led them to the normalized point count
\begin{equation}\label{eq:Ppt}
P_{a,b}^{\mathrm{pt}}(q):=
\prod_{n\geq1}(1-q^n)\cdot
\Biggl(\,\sum_{n=0}^{\infty}
\frac{\lvert M_n^{a,b}(\Fq)\rvert}{\lvert\operatorname{GL}_n(\Fq)\rvert}
\Biggr)\Bigg|_{q\mapsto q^{-1}}
=Z_{a,b}(q),
\end{equation}
the second equality obtained by substituting \(q\mapsto q^{-1}\) in
\eqref{eqn:HJO} and then multiplying by \(\prod_{n\geq1}(1-q^n)\).
This is our only use of \eqref{eqn:HJO}.

Huang--Jiang--Oblomkov conjectured that \(Z_{a,b}(q)\) is an explicit
infinite product, giving a doubly-infinite family of
Rogers--Ramanujan type identities.  Define the \defn{\(a,b\)-charge} at
\(i\in\Z\) by
\begin{equation}\label{eq:charge}
r_{a,b}(i):=
\min\{a,b,\dist(ai,(a+b)\Z)\}-1
+\ind_{(a+b)\Z}(i),
\end{equation}
where \(\dist(m,L\Z):=\min_{t\in\Z}\abs{m-Lt}\), and put
\begin{equation}\label{eq:Pab}
P_{a,b}(q):=\prod_{i\geq1}(1-q^i)^{-r_{a,b}(i)}.
\end{equation}

The charge is periodic modulo \(a+b\) and even, since it sees \(i\)
only through \(\dist(ai,(a+b)\Z)\).  It is nonnegative because
\(\gcd(a,a+b)=\gcd(a,b)=1\) forces \(\dist(ai,(a+b)\Z)\geq1\) off the
zero class, where the indicator corrects the \(-1\).  The classes
 \(\pm i\pmod{a+b}\)
carry equal exponents.  
This pairing means that \(P_{a,b}(q)\) can be expressed as a
product assembled from Dedekind's eta-function and Jacobi theta quotients of level \(a+b\).  Hence, with \(q:=e^{2\pi i\tau}\), \(P_{a,b}(q)\)
is a modular function with character on \(\Gamma(a+b)\) up to a rational power of
\(q\) (see, e.g., \cite[Chapters~2--3]{KubertLang}).

\begin{conjecture}[{\bf HJO} \cite{HJO}]
For coprime \(a,b>1\), we have $Z_{a,b}(q)=P_{a,b}(q)$.
\end{conjecture}

The conjecture is best read as indexed by \(a\).  For \(a=1\) the gap
set is empty and the identity is trivial.  For \(a=2\), coprimality
forces \(b=2k+1\), and \(Z_{2,b}(q)\), \(P_{2,b}(q)\) are the sum and
product sides of the Andrews--Gordon identity \eqref{eq:AG} with
modulus \(2k+3\) and \(i=k+1\). That layer is thus established.  At
\(b=3\) it is the first Rogers--Ramanujan identity
\[
Z_{2,3}(q)=\sum_{n\geq0}\frac{q^{n^2}}{(q)_n}
=\prod_{n\geq0}\frac{1}{(1-q^{5n+1})(1-q^{5n+4})}.
\]
Geometrically, this is the HJO identity
for the cubic \(X^2=Y^3\).  So the conjecture proposes that
much of the RR/AG phenomenon flows from the geometry of \(X^a=Y^b\).

%The \(a=3\) layer is the subject of this paper.  A recent preprint
%\cite{HOP} settles \(b=4,5,7,8\) through a finer family of
%\emph{sum-to-sum identities}, the fixed-boundary identities
%\eqref{eq:sum-plus}--\eqref{eq:sum-minus} below, which give
%\(Z_{3,b}=P_{3,b}\). For
%\(b=4,5,7\), these were proved by a delicate \(q\)-Vandermonde
%``split-and-merge'' argument; \(b=8\) required creative telescoping and
%noncommutative Gr\"obner bases handled by significant machine computation. 

The purpose of the present paper is to prove the $a=3$ layer completely.

\begin{theorem}\label{thm:main}
For every integer \(b>3\) coprime to \(3\), the following are true.

\noindent (1) We have that
$$
P_{3,b}^{\mathrm{pt}}(q)=Z_{3,b}(q)=P_{3,b}(q).
$$

\noindent
(2) We have that
\[
\sum_{\boldsymbol n\in\Z^{\Gap}\cap\mathcal C_{3,b}}
q^{Q_{3,b}(\boldsymbol n)}
\prod_{i\in\Gap}
\frac{(q)_{n_i-n_{i-3-b}}}{(q)_{n_i-n_{i-3}}(q)_{n_i-n_{i-b}}}
=\prod_{i\geq1}\bigl(1-q^i\bigr)^{-r_{3,b}(i)}.
\]

\noindent
(3) The common $q$-series in part (1) is a finite product of Jacobi theta
quotients of level \(b+3\).  Up to multiplication by a rational power
of \(q\), it is a modular function with character on \(\Gamma(b+3)\).
\end{theorem}

\begin{remark}\label{rem:fermionic}
In Section~\ref{sec:proof}, we write \(Z_{3,b}(q)\) as a sum over \(r_1\geq\cdots\geq r_k\geq0\) and \(m_1,m_2,\ldots\geq0\)
obeying inequalities, of a monomial in \(q\) times Gaussian
polynomials \eqref{eq:gaussian-def}; see the displays in the proof of
Theorem~\ref{thm:main} in \S\ref{subsec:product}, which match the sum
sides of Warnaar's identities in Proposition~\ref{prop:warnaar}.
\end{remark}

\begin{remark}[AxiomProver Certificate]
A formal certificate for Theorem~\ref{thm:main} (2-3) can be found in
\begin{center}
\url{https://github.com/AxiomMath/RR_a3}
\end{center}
AxiomProver, an AI system currently under development, was used to generate this certificate. The system verified these results in Lean assuming existing literature (i.e. the results in \cite{Schilling2001, SchillingWarnaar, WarnaarHL, WarnaarAG}).
\end{remark}

\begin{example}
We consider \((a,b)=(3,10)\), where \(f_{3,10}=17\), the period is \(13\), and
$\Gap_{3,10}=\{1,2,4,5,7,8,11,14,17\}$.
Here \eqref{def:quadratic-form} becomes
\[
Q_{3,10}(\boldsymbol n)
=\sum_{i\in\Gap_{3,10}}n_i^2
+n_1n_2+n_2n_4+n_4n_5+n_5n_7+n_7n_8
-n_1n_{11}-n_2n_{14}-n_4n_{14}-n_5n_{17}-n_7n_{17},
\]
and \(\Z^{\Gap}\cap\mathcal C_{3,10}\) is cut out by
\[
n_{17}\geq n_{14}\geq n_{11}\geq n_8\geq n_5\geq n_2\geq0,
\qquad
n_7\geq n_4\geq n_1\geq0,
\]
\[
n_{17}\geq n_7,
\qquad
n_{14}\geq n_4,
\qquad
n_{11}\geq n_1 .
\]
With \(\sum_{\boldsymbol n}\) over these tuples,
Theorem~\ref{thm:main} gives
\begin{displaymath}
\sum_{\boldsymbol n}q^{Q_{3,10}(\boldsymbol n)}\,
\frac{(q)_{n_{11}}\,(q)_{n_{14}-n_{1}}\,(q)_{n_{17}-n_{4}}}
     {(q)_{n_{11}-n_{1}}(q)_{n_{14}-n_{4}}(q)_{n_{17}-n_{7}}}
\prod_{i\in\Gap_{3,10}}\frac{1}{(q)_{n_{i}-n_{i-3}}}
=\frac{(q^{13};q^{13})_{\infty}^{2}}{(q;q)_{\infty}^{2}}\,
\theta(q^{4};q^{13})^{2}\,\theta(q^{5};q^{13}).
\end{displaymath}
\end{example}
\medskip

Answering a question of Y.~Huang, we give a combinatorial interpretation for $Z_{3,b}(q)$.

\begin{corollary}[Tableau model]
\label{cor:tableau-model}
Let \(b>3\) be coprime to \(3\), write \(b=3k\pm1\), and let
\(r_1\geq\cdots\geq r_k\geq0\).  Set
\[
(\ell,s):=(k,r_k)\ \text{ if }b=3k+1,
\qquad
(\ell,s):=(k-1,r_k)\ \text{ if }b=3k-1,
\]
let \(\mu\) be the partition with \(\mu'=(r_1+s,\ldots,r_\ell+s)\), of
size \(\abs\mu=\sum_{i=1}^{\ell}(r_i+s)\), and let \(\mathcal T_M(\mu)\) be the tuples of one-row semistandard
tableaux with entries in \(\{1,2\}\), row lengths the parts of \(\mu\)
in any fixed order, and exactly \(M\) entries equal to \(1\).  Then
\[
Z_{3,b}(q)=
\sum_{r_1\geq\cdots\geq r_k\geq0}
\frac{1}{(q)_{r_1}}
\left(\prod_{i=1}^{k-1}\qbinom{r_i}{r_{i+1}}\right)
\sum_{M\geq0}\ \sum_{T\in\mathcal T_M(\mu)}
q^{\,\sum_{i=1}^{k}r_i^2+sM-n(\mu)+\coinv(T)},
\]
where \(\coinv(T)=n(\mu)-\inv(T)\geq0\) and
\(n(\mu)=\sum_{i=1}^{\ell}\binom{r_i+s}{2}\).  Although \(\inv\)
depends on that order, the inner sum does not, by
Proposition~\ref{prop:tableau-supernomial}.
\end{corollary}

\noindent
Here \(\inv\) is Schilling's inversion statistic on tuples of one-row
tableaux and \(n(\mu)\) is as in \eqref{eq:supernomial-tilde}. Both are recalled in
\S\ref{subsec:schilling}.

\begin{remark}\label{rem:strata}
The contribution to \(Z_{3,b}(q)\) of each fixed boundary chain
\(\boldsymbol r\), denoted \(S_b(\boldsymbol r)\) in
\S\ref{subsec:product}, is a sum of \(A_1\) supernomials up to explicit
monomials:
\[
S_b(\boldsymbol r)=
\sum_{M\geq0}q^{\,\sum_{i=1}^{k}r_i^2+sM-n(\mu)}
S_{(M,\abs\mu-M),\mu}(q).
\]
The monomial may have a negative exponent: positivity is a property of
the assembled series, not of the individual factors (see
Remark~\ref{rem:laurent}).
\end{remark}

\subsection*{New ideas}

The Rogers--Ramanujan and Andrews--Gordon identities count partitions
whose parts obey difference conditions, and a difference is taken along
a linear order.  Their sum sides record exactly this: the index set is a
single decreasing chain \(r_1\geq\cdots\geq r_k\geq0\), and each factor
\((q)_{r_i-r_{i+1}}\) measures one step of it.  The series
\(Z_{a,b}(q)\) is not of this shape.  Its variables are indexed by the
gap set \(\Gap_{a,b}\) of the numerical semigroup \(\langle a,b\rangle\)
under divisibility, a Young diagram in the \(a\times b\) rectangle,
not a chain, and its factors \((q)_{n_i-n_{i-a}}\),
\((q)_{n_i-n_{i-b}}\), \((q)_{n_i-n_{i-a-b}}\) take differences in two
directions and along the diagonal.  For \(a=3\) the diagram is two
arithmetic progressions, interlacing in one order when
\(b\equiv1\pmod3\) and in the other when \(b\equiv-1\), so the two
residue classes look like separate problems.  Our contribution is a bridge
turning \(Z_{3,b}(q)\) into Warnaar's \(A_2\) Andrews--Gordon series
uniformly in \(b\), by the route

\begin{displaymath}
\begin{split}
Z_{3,b}
\ \xrightarrow[\text{Lemma~\ref{lem:Q-reindex}}]{\text{reindex}}\ &
\text{boundary term}\times\text{finite Laurent polynomial}\\
&\hskip1in \xrightarrow[\text{Thm~\ref{thm:master}}]{\text{supernomial}}\
\text{Warnaar summand}
\ \xrightarrow[\text{Lem~\ref{lem:reassemble},\ \S\ref{subsec:product}}]{\text{reassemble}}\
P_{3,b},
\end{split}
\end{displaymath}
whose arrows rest on the following ideas.

First, we introduce coordinates adapted to the two arithmetic
progressions comprising \(\Gap_{3,b}\).  Fixing the variables on the
longer one leaves two interlacing chains of free variables, in which
each fixed-boundary contribution to \(Z_{3,b}(q)\) becomes a boundary
term times a finite Gaussian-polynomial sum (Lemma~\ref{lem:Q-reindex},
\eqref{eq:Baz}, Lemma~\ref{lem:reassemble}).  This isolates all
\(b\)-dependence, apart from a boundary sum, in a single finite Laurent
polynomial.

Secondly, the heart of the paper: we recognize that finite Laurent
polynomial, for \emph{both} congruence classes at once, as a
fixed-content coefficient of an \(A_1\) completely symmetric
\(q\)-supernomial coefficient in the sense of Schilling and Warnaar.
The two exponents \(E^-\) and \(E^+\) attached to \(b\equiv\pm1\) come
from the \emph{two orders} of the same two blocks of one-row tableaux,
and the order-independence of the multitableau expansion
(Proposition~\ref{prop:tableau-supernomial}) makes the two sums equal.  After
the substitutions and \(q\)-reciprocity of \S\ref{subsec:rigged} and
summation over the fixed-content slices, this produces the inner
fermionic sums of Warnaar's \(A_2\) identities.  The resulting
transformation (Theorem~\ref{thm:master}) is of independent interest
and mentions no numerical semigroups.

Thirdly, we restore the boundary summation, match the result to
Warnaar's \(A_2\) Andrews--Gordon identities, and identify the
resulting theta products with the HJO charge product \(P_{3,b}(q)\) by
a residue-class calculation.

Section~\ref{sec:schilling-warnaar} gives the
Schilling--Warnaar theory of completely symmetric \(q\)-supernomial
coefficients: Kostka--Foulkes polynomials, the multitableau
(inversion) expansion, the rigged-configuration expansion, and
Warnaar's \(A_2\) identities.  A reader familiar with
this material can skim it.  Section~\ref{sec:proof} proves
Theorem~\ref{thm:main}, by reindexing \(Q_{3,b}\), the fixed-boundary
supernomial transformation with auxiliary identities, the reassembly of
\(Z_{3,b}\), and the passage to the product.

\section*{Acknowledgements}\noindent
The authors thank George Andrews, Yifeng Huang, Seewoo Lee,
Peter Paule, and Ashvin Swaminathan for discussions related to this work.  We are especially
grateful to Yifeng Huang for raising the question of a combinatorial
interpretation of the sum side of \(Z_{a,b}(q)\) beyond the classical
case \(a=2\). The tableau model of
Corollary~\ref{cor:tableau-model} is our answer for \(a=3\).

\section{The Schilling--Warnaar theory of \texorpdfstring{$q$}{q}-supernomials}
\label{sec:schilling-warnaar}

We recall the completely symmetric \(q\)-supernomial coefficients of
Schilling and Warnaar, their two combinatorial expansions, and
Warnaar's \(A_2\) Andrews--Gordon identities
\cite{Schilling2001,SchillingWarnaar,WarnaarAG}.
A \defn{partition} \(\lambda=(\lambda_1\geq\lambda_2\geq\cdots\geq0)\)
has finitely many nonzero terms, its \defn{parts}, and finite sum
\(\abs\lambda=\sum_i\lambda_i\).  Its \defn{Young diagram} is the
left-justified array with \(\lambda_i\) cells in row \(i\), rows
numbered from the top; the \defn{conjugate} \(\lambda'\) is the
transpose, so \(\lambda'_j=\#\{i:\lambda_i\geq j\}\).  A \defn{weak
composition} is a finite, not necessarily ordered, sequence of
nonnegative integers, its \defn{size} again the sum of its entries.

A \defn{semistandard (Young) tableau} of \defn{shape} \(\eta\) fills
the Young diagram of \(\eta\) with positive integers, \emph{weakly
increasing along each row} and \emph{strictly increasing down each
column}.  Its \defn{content} is the weak composition
\(\lambda=(\lambda_1,\lambda_2,\ldots)\), where \(\lambda_c\) counts
the entries equal to \(c\).  A one-row tableau with entries in
\(\{1,2\}\) is therefore determined by its number of \(1\)'s; this is
the only case we use.  The number of semistandard tableaux of shape
\(\eta\) and content \(\lambda\) is the \defn{Kostka number}
\(K_{\eta\lambda}\in\Z_{\geq0}\), which depends on \(\lambda\) only
through the multiset of its parts.  See Figure~\ref{fig:ssyt}.

\begin{figure}[H]
\centering
\begin{tikzpicture}[line join=round,
  tcell/.style={draw, minimum size=0.72cm, anchor=south west,
    font=\small, inner sep=0pt, fill=onefill!45}]
\def\cs{0.72}
% shape eta = (3,2,1), a semistandard filling, content (1,2,2,1):
% row 1 (bottom-to-top in south-west coords is top row highest y)
% top row:    1 2 2
% middle row:  3 3
% bottom row:   4
\node[tcell] at (0*\cs,2*\cs) {$1$};
\node[tcell] at (1*\cs,2*\cs) {$2$};
\node[tcell] at (2*\cs,2*\cs) {$2$};
\node[tcell] at (0*\cs,1*\cs) {$3$};
\node[tcell] at (1*\cs,1*\cs) {$3$};
\node[tcell] at (0*\cs,0*\cs) {$4$};
\end{tikzpicture}
\caption{A semistandard tableau of shape \(\eta=(3,2,1)\) and content
\(\lambda=(1,2,2,1)\).}
\label{fig:ssyt}
\end{figure}

The \defn{Gaussian polynomial} (or \(q\)-binomial coefficient) is
\begin{equation}\label{eq:gaussian-def}
\qbinom Nj=
\begin{cases}
\dfrac{(q)_N}{(q)_j(q)_{N-j}},&0\leq j\leq N,\\[6pt]
0,&\text{otherwise},
\end{cases}
\qquad\text{where }(q)_m=\prod_{t=1}^m(1-q^t).
\end{equation}
Since \(\qbinom Nj\) vanishes outside \(0\leq j\leq N\), every finite
sum below may be taken over all integers, the Gaussian polynomials
imposing the inequalities.  We use freely
\begin{displaymath}
\qbinom Nj=\sum_{\substack{w\in\{0,1\}^N\\\abs w=j}}q^{\inv(w)}
\qquad {\text {\rm and}} \qquad
\qbinomQ Nj{q^{-1}}=q^{-j(N-j)}\qbinom Nj,
\end{displaymath}
standard \(q\)-counting of lattice paths and \(q\)-reciprocity
respectively, where \(\abs w\) is the number of \(1\)'s in \(w\) and
\(\inv(w)=\#\{p<p':w_p=1,\ w_{p'}=0\}\).

\subsection{The Schilling supernomial and its two expansions}
\label{subsec:schilling}

Supernomial coefficients package, in a single symmetric-function
object, the fermionic sums on the sum side of Rogers--Ramanujan-type
identities; see \cite{WarnaarSurvey} for a survey and the open problems
around it.  Let \(\mu\) be a partition and \(\lambda\) a weak
composition with \(\abs\lambda=\abs\mu\).  Following
\cite[\S2]{SchillingWarnaar}, the \defn{completely symmetric
\(q\)-supernomial coefficient} is
\begin{equation}\label{eq:supernomial-def}
S_{\lambda\mu}(q):=\sum_{\eta}K_{\eta\lambda}\,K_{\eta\mu}(q),
\end{equation}
summed over partitions \(\eta\) of \(\abs\mu\), where
\(K_{\eta\lambda}\in\Z_{\geq0}\) is the Kostka number and
\(K_{\eta\mu}(q)\) the Kostka--Foulkes polynomial, its
\(q\)-analogue arising as a Hall--Littlewood transition coefficient.
Since \(K_{\eta\lambda}\) depends only on the multiset of parts of
\(\lambda\), so does \(S_{\lambda\mu}(q)\).  We work with the
normalized (``tilde'') version
\begin{equation}\label{eq:supernomial-tilde}
n(\mu):=\sum_{p<p'}\min(\mu_p,\mu_{p'}),
\qquad
\widetilde S_{\lambda\mu}(q):=q^{n(\mu)}S_{\lambda\mu}(q^{-1}).
\end{equation}

Then \(\widetilde S_{\lambda\mu}\) is again symmetric under permuting
the parts of \(\lambda\), its \defn{content symmetry}.  This is
immediate from \eqref{eq:supernomial-def}, and we use it only to match
the index convention of Proposition~\ref{prop:SchillingA1}.  A second
symmetry, the order-independence of the multitableau expansion
(Proposition~\ref{prop:tableau-supernomial}), permutes instead the row
components built from \(\mu\).  It does not follow from the definition,
and it is this deeper symmetry that later reconciles the two congruence
classes \(b\equiv\pm1\pmod3\); see
Remark~\ref{rem:order-vs-content}.

The first expansion writes \(\widetilde S_{\lambda\mu}\) as an
inversion generating function over tuples of one-row tableaux.  Fix an
ordering of the parts of \(\mu\), and let \(T=(T_1,T_2,\dots)\) be a
tuple of one-row semistandard tableaux with entries in \(\{1,2\}\),
where \(T_p\) has length the \(p\)-th part of \(\mu\) and the total
content is \(\lambda\).  Schilling's statistic \(\inv(T)\) depends on
that ordering: an inversion is a same-column pair with a \(2\) in the
earlier component and a \(1\) in the later, or an adjacent-column pair
with a \(1\) in column \(i+1\) of the earlier component and a \(2\) in
column \(i\) of the later.

\begin{proposition}[Schilling \cite{Schilling2001}]
\label{prop:tableau-supernomial}
With the notation above,
\begin{equation}\label{eq:tableau-supernomial}
\widetilde S_{\lambda\mu}(q)=\sum_T q^{\inv(T)}
\end{equation}
for \emph{every} ordering of the parts of \(\mu\).  In particular the
right-hand side is independent of that ordering, although the individual
exponents \(\inv(T)\) are not.
\end{proposition}

\begin{remark}
\label{rem:order-vs-content}
Proposition~\ref{prop:tableau-supernomial} says that permuting the rows
leaves the polynomial unchanged while changing the individual inversion
exponents; this is the freedom we exploit in Section~\ref{sec:proof} to
produce \(E^-\) and \(E^+\).  It is easily conflated with the
\emph{content symmetry} of \eqref{eq:supernomial-tilde}, but the two
differ in exactly the way that matters here.  Content symmetry permutes
the parts of the \emph{lower} index \(\lambda\); it is immediate from
\eqref{eq:supernomial-def}, since \(K_{\eta\lambda}\) depends only on
the multiset of parts of \(\lambda\), and we use it only to match
Schilling's index convention in Proposition~\ref{prop:SchillingA1}.
Order-independence instead permutes the row lengths taken from
\(\mu\); it does \emph{not} follow from \eqref{eq:supernomial-def}, but
requires Schilling's statistic-preserving bijection
\cite[Theorem~4.1]{Schilling2001}.  It is the latter that reconciles the
classes \(b\equiv\pm1\pmod3\).
\end{remark}

\begin{proof}[Proof of Proposition~\ref{prop:tableau-supernomial}]
This is \cite[\S3.2, \S4.1]{Schilling2001}.  Order-independence follows
from Schilling's statistic-preserving bijection
\cite[Theorem~4.1]{Schilling2001}, whose target is a set of rigged
configurations with parameter \(\mu\), the partition of the multiset of
row lengths; it retains no memory of the component order, so permuting
the rows gives the same polynomial with different inversion exponents.
\end{proof}

The second expansion is fermionic and comes from the
rigged-configuration bijection.  A word on the labels \(A_1\) and
\(A_2\).  In the crystal-bases and rigged-configuration literature,
type \(A_{n-1}\) refers to objects built on an \(n\)-letter alphabet
\(\{1,\ldots,n\}\), matching the natural representation of
\(\mathfrak{sl}_n\) \cite{HongKang,KirillovSchillingShimozono}.  So the
\defn{\(A_1\) case} is the two-letter case, tableaux filled from
\(\{1,2\}\) with content \(\lambda=(\lambda_1,\lambda_2)\), and the
\defn{\(A_2\) case} the three-letter case on \(\{1,2,3\}\).  Our
supernomial computations are entirely \(A_1\), while the products we
invoke are Warnaar's \(A_2\) Andrews--Gordon identities of
\S\ref{subsec:warnaar}: the argument is a bridge from \(A_1\)
supernomial algebra to an \(A_2\) product.

We now state the \(A_1\) expansion in the specialized form we shall
use.

\begin{proposition}[Schilling \cite{Schilling2001}, \texorpdfstring{$A_1$}{A1} specialization]
\label{prop:SchillingA1}
Let \(\mu'=(c_1,\ldots,c_\ell)\) with
\(c_1\geq\cdots\geq c_\ell\geq0\) be the conjugate of \(\mu\), and set
\(c_{\ell+1}:=0\).  For an integer \(M\) put
\(\lambda=(\abs\mu-M,M)\).  Then we have
\begin{equation}\label{eq:SchillingA1}
\widetilde S_{(\abs\mu-M,M),\mu}(q)
=\sum_{\substack{m_1,\ldots,m_\ell\geq0\\ \sum_i m_i=M}}
q^{\Phi(\boldsymbol m)}\prod_{i=1}^{\ell}\qbinom{T_i}{m_i},
\end{equation}
where \(m_{\ell+1}:=0\), $T_{\ell}:=c_{\ell}$ and
\begin{align*}
&T_i:=c_i-c_{i+1}+m_{i+1}\ \ (1\leq i<\ell),\\
&\Phi(\boldsymbol m):=\sum_{i=1}^{\ell-1}(c_{i+1}-m_{i+1})m_i.
\end{align*}
The identity holds for either order of the two entries of
\(\lambda=(\abs\mu-M,M)\). 
\end{proposition}

\begin{proof}
This is the two-part-content specialization of Schilling's formula
\cite[(2.1)--(2.3)]{Schilling2001}. We spell it out because the
literature contains several charge/cocharge normalizations of
Kostka--Foulkes and supernomial polynomials.  In
\cite[(2.1)--(2.3)]{Schilling2001}, the case of a two-part content
has one intermediate partition
\(\nu\subseteq\mu'=(c_1,\ldots,c_\ell)\), with
\(\abs\nu=\abs\mu-M\).  Put \(m_i=c_i-\nu_i\).  Then we have
\(\sum_i m_i=M\), and the \(i\)-th \(q\)-multinomial factor in
\cite[(2.1)]{Schilling2001} becomes
\[
\genfrac{[}{]}{0pt}{}{c_i-\nu_{i+1}}
 {\nu_i-\nu_{i+1},\ c_i-\nu_i}_{\!q}
 =\qbinom{c_i-c_{i+1}+m_{i+1}}{m_i}
 =\qbinom{T_i}{m_i}.
\]
Moreover, Schilling's exponent \cite[(2.3)]{Schilling2001} becomes
\[
\sum_{i=1}^{\ell-1}\nu_{i+1}(c_i-\nu_i)
 =\sum_{i=1}^{\ell-1}(c_{i+1}-m_{i+1})m_i
 =\Phi(\boldsymbol m).
\]
This derives \eqref{eq:SchillingA1} directly, with no additional
normalization factor; the formula uses precisely Schilling's
convention
\(\widetilde S_{\lambda\mu}(q)=q^{n(\mu)}S_{\lambda\mu}(q^{-1})\)
from \eqref{eq:supernomial-tilde}.  (A computation made with a
modified, charge-normalized Kostka--Foulkes routine must be converted
to this convention before comparison.)  That the identity holds for
either order of the entries of \(\lambda\) is content symmetry,
\eqref{eq:supernomial-tilde}.  Finally, \(0\leq m_i\leq T_i\) is
equivalent to \(\nu_i\geq\nu_{i+1}\) and \(\nu_i\leq c_i\), so the
vanishing convention \eqref{eq:gaussian-def} legitimizes the
unrestricted nonnegative sum; and if \(M>\abs\mu\), both sides
vanish, since \(m_\ell\leq c_\ell\) and
\(m_i\leq c_i-c_{i+1}+m_{i+1}\) give \(m_i\leq c_i\) for all \(i\),
whence \(\sum_i m_i\leq\sum_i c_i=\abs\mu\).
\end{proof}

Equating Propositions~\ref{prop:tableau-supernomial}
and~\ref{prop:SchillingA1} gives one polynomial identity with two
faces: a tableau side, whose exponent depends on the chosen ordering of
rows, and a fermionic side, manifestly a sum of products of Gaussian
polynomials.  This is the key mechanism of the paper.

\subsection{Warnaar's \texorpdfstring{$A_2$}{A2} Andrews--Gordon
identities}
\label{subsec:warnaar}
The final ingredient is the pair of \(A_2\) analogues of the
Andrews--Gordon identities proved by Warnaar
\cite[(1.10)--(1.11)]{WarnaarAG}.  These express certain
fermionic double sums as products of theta functions.  For
\(x,Q\) with \(\abs Q<1\), write
\begin{equation}\label{eq:theta-def}
(x;Q)_\infty:=\prod_{j\geq0}(1-xQ^j)\qquad {\text {\rm and}}\qquad
\theta(x;Q):=(x;Q)_\infty(Q/x;Q)_\infty,
\end{equation}
and \(\theta(x,y,z;Q):=\theta(x;Q)\theta(y;Q)\theta(z;Q)\).

\begin{proposition}[Warnaar \cite{WarnaarAG}]
\label{prop:warnaar}
For every integer \(k\geq2\), we have
\begin{align}
&\sum_{\substack{r_1,\ldots,r_k\geq0\\ m_1,\ldots,m_{k-1}\geq0}}
\frac{q^{\,r_k^2+\sum_{i=1}^{k-1}(r_i^2-r_im_i+m_i^2)}}{(q)_{r_1}}
\left(\prod_{i=1}^{k-1}\qbinom{r_i}{r_{i+1}}\right)
\qbinom{r_{k-1}+r_k}{m_{k-1}}
\prod_{i=1}^{k-2}\qbinom{r_i-r_{i+1}+m_{i+1}}{m_i}
\notag\\
&\qquad\qquad
=\frac{(q^{3k+2};q^{3k+2})_\infty^2}{(q;q)_\infty^2}\,
\theta\!\left(q^k,q^{k+1},q^{k+1};q^{3k+2}\right),
\label{eq:WarnaarMinus}
\end{align}
which is \cite[(1.10)]{WarnaarAG}, and, for every integer
\(k\geq1\), we have
\begin{align}
&\sum_{\substack{r_1,\ldots,r_k\geq0\\ m_1,\ldots,m_k\geq0}}
\frac{q^{\sum_{i=1}^{k}(r_i^2-r_im_i+m_i^2)}}{(q)_{r_1}}
\left(\prod_{i=1}^{k-1}\qbinom{r_i}{r_{i+1}}\right)
\qbinom{2r_k}{m_k}
\prod_{i=1}^{k-1}\qbinom{r_i-r_{i+1}+m_{i+1}}{m_i}
\notag\\
&\qquad\qquad
=\frac{(q^{3k+4};q^{3k+4})_\infty^2}{(q;q)_\infty^2}\,
\theta\!\left(q^{k+1},q^{k+1},q^{k+2};q^{3k+4}\right),
\label{eq:WarnaarPlus}
\end{align}
which is \cite[(1.11)]{WarnaarAG} (his parameter replaced by
\(k+1\)).
\end{proposition}

\noindent
After the substitutions and \(q\)-reciprocity of
\S\ref{subsec:rigged}, summing the fixed-\(M\) expansion
\eqref{eq:SchillingA1} over \(M\) gives the inner sums in \eqref{eq:WarnaarMinus} and \eqref{eq:WarnaarPlus}.  This allows us to assemble \(Z_{3,b}\) into Warnaar's products.

\section{Proof of the main theorem}
\label{sec:proof}

The proof of Theorem~\ref{thm:main} has four steps: coordinates on the
gap set and the statement of the fixed-boundary identities; the
reindexing of the semigroup quadratic form
(\S\ref{subsec:quadratic}); a single fixed-boundary supernomial
transformation yielding both congruence classes
(\S\ref{subsec:master}--\S\ref{subsec:proof-fixed}); and the restoration
of the boundary summation and passage to Warnaar's products
(\S\ref{subsec:product}).

Let \(b>3\) be coprime to \(3\), let \(f=2b-3\), and fix
\(
r_1\geq r_2\geq\cdots\geq r_k\geq0,
\)
where \(b=3k\pm1\).  In the sums over \(\boldsymbol n\), impose
\begin{equation}\label{eq:fixed-r}
n_{f+3-3j}=r_j\qquad(1\leq j\leq k).
\end{equation}
The indices \(f+3-3j=2b-3j\) lie in \(\Gap_{3,b}\) and successive ones
differ by \(3\), so the inequalities on the \(r_j\) are the
boundary-chain inequalities inherited from \(\mathcal C_{3,b}\).  Every
\(\sum_{\boldsymbol n}\) in Theorem~\ref{thm:fixed-boundary} runs over
all \(\boldsymbol n\in\Z^{\Gap_{3,b}}\) satisfying \eqref{eq:fixed-r},
with \(n_t=0\) for \(t\notin\Gap_{3,b}\) as before.  The Gaussian
factors make these sums finite: their nonzero terms confine every free
coordinate to an interval between \(0\) and some \(r_j\); in the
coordinates of Lemma~\ref{lem:Q-reindex}, the nonzero terms are
exactly those satisfying \eqref{eq:domain}.

The next theorem is the pivot of the proof.  On the left stands the
contribution to \(Z_{3,b}\) of all \(\boldsymbol n\) with a fixed
boundary chain \(\boldsymbol r=(r_1,\ldots,r_k)\). On the right, the
 double sum that is the summand of Warnaar's \(A_2\)
Andrews--Gordon identities
\eqref{eq:WarnaarMinus}--\eqref{eq:WarnaarPlus}.  We defer the proof to
\S\ref{subsec:proof-fixed}.

\begin{theorem}
\label{thm:fixed-boundary}
The following identities hold.

\smallskip
\noindent
\emph{(i) If \(b=3k+1\), then we have}
\begin{align}
&\sum_{\boldsymbol n}q^{Q_{3,b}(\boldsymbol n)}
\prod_{j=1}^{k}\left(
\qbinom{n_{3j+2}}{n_{3j-1}}
\qbinom{r_{k-j+1}-n_{3j-5}}
{n_{3j-2}-n_{3j-5}}\right)
\notag\\
&\qquad =
\sum_{m_1,\ldots,m_k\geq0}
q^{\sum_{i=1}^{k}(r_i^2-r_i m_i+m_i^2)}
\qbinom{2r_k}{m_k}
\prod_{i=1}^{k-1}
\qbinom{r_i-r_{i+1}+m_{i+1}}{m_i}.
\label{eq:sum-plus}
\end{align}

\smallskip
\noindent
\emph{(ii) If \(b=3k-1\) (so that \(b\geq5\) and \(k\geq2\)), then we have}
\begin{align}
&\sum_{\boldsymbol n}q^{Q_{3,b}(\boldsymbol n)}
\prod_{j=1}^{k-1}\left(
\qbinom{n_{3j+1}}{n_{3j-2}}
\qbinom{r_{k-j}-n_{3j-4}}
{n_{3j-1}-n_{3j-4}}\right)
\notag\\
&\qquad =
\sum_{m_1,\ldots,m_{k-1}\geq0}
q^{\,r_k^2+
\sum_{i=1}^{k-1}(r_i^2-r_i m_i+m_i^2)}
\qbinom{r_{k-1}+r_k}{m_{k-1}}
\prod_{i=1}^{k-2}
\qbinom{r_i-r_{i+1}+m_{i+1}}{m_i}.
\label{eq:sum-minus}
\end{align}
\end{theorem}

\subsection{The semigroup quadratic form for
\texorpdfstring{\(a=3\)}{a=3}}
\label{subsec:quadratic}

For \(a=3\) the function \(U_{3,b}\) of \eqref{def:quadratic-form} is
simple:
\begin{equation}\label{eq:U3}
U_{3,b}(x)=
\begin{cases}
1,&0\leq x\leq2,\\
-1,&b\leq x\leq b+2,\\
0,&\text{otherwise}.
\end{cases}
\end{equation}
Therefore, \(Q_{3,b}\) is the sum of the squares, the products of gap
variables whose indices differ by \(1\) or \(2\), and the negatives of
those whose indices differ by \(b\), \(b+1\), or \(b+2\).  The gap sets
are
\begin{align}
\Gap_{3,3k-1}
&=\{3t+1:0\leq t\leq2k-2\}
\mathbin{\dot\cup}
\{3t+2:0\leq t\leq k-2\},\label{eq:gaps-minus}\\
\Gap_{3,3k+1}
&=\{3t+1:0\leq t\leq k-1\}
\mathbin{\dot\cup}
\{3t+2:0\leq t\leq2k-1\}.
\label{eq:gaps-plus}
\end{align}
Multiples of \(3\) are not gaps.  For \(b=3k-1\) the least
element of \(\Gamma\) congruent to $2\pmod 3$ is \(b\) and the
least one for \(1\) is \(2b\), giving \eqref{eq:gaps-minus}. For
\(b=3k+1\), the classes \(1\) and \(2\) switch roles, giving
\eqref{eq:gaps-plus}.

The two progressions in \eqref{eq:gaps-minus} and \eqref{eq:gaps-plus}
interlace in opposite orders.  We introduce coordinates retaining this
distinction only in the exponent.  Let \(\ell\geq1\), let
\(\boldsymbol r=(r_1,\ldots,r_\ell)\) satisfy
\(r_1\geq\cdots\geq r_\ell\geq0\), and let \(s\geq0\).  Set
\[
a_0:=s,\qquad a_{\ell+1}:=0,\qquad z_{\ell+1}:=0,
\]
write \(\boldsymbol a=(a_1,\ldots,a_\ell)\) and
\(\boldsymbol z=(z_1,\ldots,z_\ell)\), and define
\begin{align}
E^{-}(\boldsymbol a,\boldsymbol z)
&:=\sum_{i=1}^{\ell}
\bigl(a_i^2+z_i^2+a_i z_i+z_i a_{i-1}
-r_i(a_i+z_i)\bigr),\label{eq:Eminus}\\
E^{+}(\boldsymbol a,\boldsymbol z)
&:=\sum_{i=1}^{\ell}
\bigl(a_i^2+z_i^2+a_i z_i+a_{i+1}z_i
-r_i(a_{i+1}+z_i)\bigr).
\label{eq:Eplus}
\end{align}
Thus \(E^{+}\) is obtained from \(E^{-}\) by raising both \(a_{i-1}\)
and \(a_i\) to \(a_{i+1}\): the coupling \(z_ia_{i-1}\) becomes
\(a_{i+1}z_i\), and the linear term \(-r_i(a_i+z_i)\) becomes
\(-r_i(a_{i+1}+z_i)\), while every quadratic and \(z\)-linear term is
unchanged.  This is exactly the reversal of the two tableau blocks
(\S\ref{subsec:multitableau}), and it reconciles the classes
\(b\equiv\pm1\pmod3\).  In particular the linear coefficient in
\(E^{+}\) is \(-r_ia_{i+1}\), \emph{not} \(-r_ia_i\);
Theorem~\ref{thm:master} is false for \(E^{+}\) with the latter choice.

\begin{lemma}\label{lem:Q-reindex}
The following formulas hold.

\begin{enumerate}
\item If \(b=3k-1\), take \(\ell=k-1\), \(s=r_k\), and
\[
a_i:=n_{3(k-i)-2}\quad(0\leq i\leq\ell),\qquad
z_i:=n_{3(k-i)-1}\quad(1\leq i\leq\ell).
\]
Then we have
\begin{equation}\label{eq:Qminus}
Q_{3,3k-1}(\boldsymbol n)
=\sum_{i=1}^{k}r_i^2
+E^{-}(\boldsymbol a,\boldsymbol z).
\end{equation}

\item If \(b=3k+1\), take \(\ell=k\), \(s=r_k\), and
\[
a_i:=n_{3(k-i)+2}\quad(0\leq i\leq\ell),\qquad
z_i:=n_{3(k-i)+1}\quad(1\leq i\leq\ell).
\]
Here \(a_0=r_k\).  Then we have
\begin{equation}\label{eq:Qplus}
Q_{3,3k+1}(\boldsymbol n)
=\sum_{i=1}^{k}r_i^2
+E^{+}(\boldsymbol a,\boldsymbol z).
\end{equation}
\end{enumerate}
\end{lemma}

\begin{proof}
We record the indices and the resulting index differences.  For
\(b=3k-1\) and $1\leq i\leq k,$ we have
$R_i=6k-2-3i,$ 
$A_i=3(k-i)-2$, and $Z_i=3(k-i)-1$,
and the nonzero nonsquare contributions prescribed by \eqref{eq:U3}
are
\smallskip

\[
\begin{array}{|c|c|c|}
\hline
\text{pair}&\text{index difference}&\text{contribution}\\ \hline
(A_i,Z_i)&1& a_i z_i\\
(Z_i,A_{i-1})&2& z_i a_{i-1}\\
(A_i,R_i)&3k=b+1&-a_i r_i\\
(Z_i,R_i)&3k-1=b&-z_i r_i .\\ \hline
\end{array}
\]
\smallskip

\noindent
For \(b=3k+1\), we have
$R_i=6k+2-3i,$ 
$A_i=3(k-i)+2$, and $Z_i=3(k-i)+1$.
The  list is
\smallskip

\[
\begin{array}{|c|c|c|}
\hline
\text{pair}&\text{index difference}&\text{contribution}\\ \hline
(Z_i,A_i)&1&z_i a_i\\
(A_{i+1},Z_i)&2&a_{i+1}z_i\\
(Z_i,R_i)&3k+1=b&-z_i r_i\\
(A_{i+1},R_i)&3k+3=b+2&-a_{i+1}r_i ,\\ \hline
\end{array}
\]
\smallskip

\noindent
the second and fourth rows occurring only for \(i<k\), which is
encoded by \(a_{k+1}=0\).

Both lists are exhaustive.  Distinct indices within any one of the
\(A\)-, \(Z\)- or \(R\)-families differ by a multiple of \(3\), and
their ranges rule out the only relevant multiple; between the \(A\)-
and \(Z\)-families the only differences in the support of
\eqref{eq:U3} are the displayed \(1\) and \(2\).  Finally, we have
$R_j-A_i=3k+3(i-j)$ and $R_j-Z_i=3k-1+3(i-j)$
in the minus case, and
$R_j-Z_i=3k+1+3(i-j)$ and $R_j-A_{i+1}=3k+3+3(i-j)$
in the plus case, so a value in \(\{b,b+1,b+2\}\) occurs only when
\(i=j\).  Adding the squares of all free and fixed variables gives
\eqref{eq:Qminus} and \eqref{eq:Qplus}.
\end{proof}

Substitution of the same indices into the Gaussian products in
Theorem~\ref{thm:fixed-boundary} gives, in both congruence classes,
the expression
\begin{equation}\label{eq:Baz}
B_{\boldsymbol r,s}(\boldsymbol a,\boldsymbol z)
:=\prod_{i=1}^{\ell}
\qbinom{a_{i-1}}{a_i}
\qbinom{r_i-z_{i+1}}{z_i-z_{i+1}}.
\end{equation}
The nonzero terms satisfy
\begin{equation}\label{eq:domain}
s=a_0\geq a_1\geq\cdots\geq a_\ell\geq0,
\qquad
r_i\geq z_i\geq z_{i+1}\quad(1\leq i\leq\ell).
\end{equation}
Conversely, every integral pair satisfying \eqref{eq:domain}
contributes a well-defined polynomial.  Therefore,
\eqref{eq:Baz} and \eqref{eq:domain} lose no terms from either
fixed-boundary sum.

\subsection{A fixed-boundary
\texorpdfstring{\(q\)}{q}-supernomial transformation}
\label{subsec:master}

We now state the technical heart of the paper, a single 
identity accounting for both congruence classes \(b\equiv\pm1\pmod3\)
at once.  It transforms the sum built from the Gaussian
product \(B_{\boldsymbol r,s}\) of \eqref{eq:Baz} and either exponent
\(E^\pm\) of \eqref{eq:Eminus}--\eqref{eq:Eplus} into a sum of the
shape appearing on the sum side of Warnaar's \(A_2\)
Andrews--Gordon identities.  The connection to \(Z_{3,b}\) is made 
in \S\ref{subsec:proof-fixed}.

\begin{theorem}\label{thm:master}
Let \(\ell\geq1\),
\(r_1\geq\cdots\geq r_\ell\geq0\), and \(s\geq0\).  For either
choice \(E=E^{-}\) or \(E=E^{+}\), we have
\begin{align}
&\sum_{\boldsymbol a,\boldsymbol z}
q^{E(\boldsymbol a,\boldsymbol z)}
B_{\boldsymbol r,s}(\boldsymbol a,\boldsymbol z)
\notag\\
&\qquad =
\sum_{m_1,\ldots,m_\ell\geq0}
q^{\sum_{i=1}^{\ell}(m_i^2-r_i m_i)}
\qbinom{r_\ell+s}{m_\ell}
\prod_{i=1}^{\ell-1}
\qbinom{r_i-r_{i+1}+m_{i+1}}{m_i},
\label{eq:master}
\end{align}
where the sum on the left is over \eqref{eq:domain}.
Moreover, \eqref{eq:master} is true after restricting the
left side by
\[
\sum_{i=1}^{\ell}(a_i+z_i)=M
\]
and the right side by \(\sum_{i=1}^{\ell}m_i=M\), for every
\(M\geq0\).
\end{theorem}

\noindent
The fixed-\(M\) refinement is structural, and the
unrestricted identity follows by summing these slices.

\begin{remark}
\label{rem:laurent}
Both sides of \eqref{eq:master} are finite \emph{Laurent} polynomials:
the exponents \(\sum_i(m_i^2-r_im_i)\) and \(E^{\pm}\) are negative for
many admissible values.  This is intrinsic to the reciprocity
\eqref{eq:qreciprocity}, and the quadratic prefactors of
\eqref{eq:Qminus} and \eqref{eq:Qplus} restore nonnegativity on
reassembly.  A numerical check truncated at \(q^0\) will therefore
appear to falsify \eqref{eq:master}.
\end{remark}

\subsection{The multitableau expansion}
\label{subsec:multitableau}

We compute the fermionic weight
\(B_{\boldsymbol r,s}(\boldsymbol a,\boldsymbol z)\) of \eqref{eq:Baz}
as an inversion generating function over tableaux. We package the
boundary data \((\boldsymbol r,s)\) into a partition \(\mu\), realize
each admissible \((\boldsymbol a,\boldsymbol z)\) as a tuple of one-row
tableaux filled with \(1\)'s and \(2\)'s, and show that Schilling's
inversion statistic \(\inv\) reproduces the exponent \(E^{\pm}\).  This
 is entirely combinatorial: \(a_i\) and \(z_i\) count how
many rows of each block reach column \(i\), and every inversion is a
local crossing between a \(1\) and a \(2\).
Figure~\ref{fig:multitableau} shows a small instance and the inversion
rule in miniature.

Let \(\beta\) be the partition with \(\beta'=(r_1,\ldots,r_\ell)\), and
let \(\mu\) be \(\beta\) with \(s\) parts of size \(\ell\) adjoined, so
\begin{equation}\label{eq:muprime}
\mu'=(r_1+s,\ldots,r_\ell+s).
\end{equation}
Order the one-row components from \(\beta\) by weakly increasing
length, the order among equal lengths being fixed but arbitrary; the
\(s\) new rows form a second block, and we use both orders of the two
blocks.  Let \(u_p\) be the number of entries equal to \(1\) in
\(T_p\), so that \(T_p\) has \(u_p\) ones followed by twos.  Counting
the rows of each block that reach column \(i\), set
\[
a_i:=\#\{p:u_p\geq i\}\ \text{(adjoined block)},
\qquad
z_i:=\#\{p:u_p\geq i\}\ \text{(\(\beta\)-block)}.
\]
The \(\beta\)-block has \(r_i\) rows of length at least \(i\), and every
adjoined row has length \(\ell\).  Hence \((\boldsymbol a,\boldsymbol z)\)
satisfies \eqref{eq:domain}, and the total number of ones is
\begin{equation}\label{eq:Mlevels}
M=\sum_pu_p=\sum_{i=1}^{\ell}(a_i+z_i).
\end{equation}

Fix \((\boldsymbol a,\boldsymbol z)\).  At level \(i\) the \(a_i\)
continuing rows are chosen from the \(a_{i-1}\) rows that reached level
\(i-1\), and in the \(\beta\)-block the \(z_i-z_{i+1}\) rows ending
precisely at level \(i\) are chosen from the \(r_i-z_{i+1}\) rows of
length at least \(i\) that do not continue to level \(i+1\).  Applying
the binary-word interpretation
\[
\qbinom Nk=\sum_{\substack{w\in\{0,1\}^N\\\abs w=k}}
q^{\inv(w)}
\]
at every level gives statistics \(J_a,J_z\): the word for \(J_a\)
records the continuing rows in component order, that for \(J_z\) the
rows ending at the given level among the eligible ones.  By
construction,
\begin{equation}\label{eq:Jgf}
\sum_{\substack{T:\,\boldsymbol a(T)=\boldsymbol a\\
\boldsymbol z(T)=\boldsymbol z}}
q^{J_a(T)+J_z(T)}
=B_{\boldsymbol r,s}(\boldsymbol a,\boldsymbol z).
\end{equation}

We now compare \(J_a+J_z\) with the inversion statistic, the
only point where the order of the two row blocks matters.  In the
single-row specialization, an inversion is a same-column pair with a
\(2\) in the earlier component and a \(1\) in the later, or an
adjacent-column pair with a \(1\) in column \(i+1\) of the earlier
component and a \(2\) in column \(i\) of the later
\cite[Section~3.2]{Schilling2001}.  Write \(\inv_a(T)\) and
\(\inv_z(T)\) for the inversions internal to the adjoined and
\(\beta\)-blocks.

\begin{lemma}\label{lem:internal}
With the notation above, we have
\begin{equation}\label{eq:internal}
\inv_a(T)+J_a(T)=\sum_{i=1}^{\ell}a_i(s-a_i),
\qquad
\inv_z(T)+J_z(T)=\sum_{i=1}^{\ell}z_i(r_i-z_i).
\end{equation}
In both blocks the count is pairwise: two rows contribute to
\(\inv+J\) once for every level, at most the length of the shorter of
the two, that is reached by exactly one of their two strings of
\(1\)'s.
\end{lemma}

\begin{proof}
The claim is pairwise.  Fix two rows in their prescribed component
order, with \(u\) and \(v\) ones.  Up to the length of the shorter row,
same-column inversions contribute once at each level strictly above
\(u\) and at most \(v\), and adjacent-column inversions once at each
level strictly above \(v+1\) and at most \(u\).  The binary-word
statistic contributes the one missing boundary level if \(u>v\), and
nothing if \(u\leq v\).  Since the earlier row is no longer than the
later, \(\inv+J\) therefore contributes once for every level, at most
the length of the shorter row, reached by exactly one of the two strings
of \(1\)'s.  Summing over pairs gives \(\sum_i a_i(s-a_i)\) in the
adjoined block.  At level \(i\) only the \(r_i\) rows of the
\(\beta\)-block of length at least \(i\) are eligible, and exactly
\(z_i\) of them reach it; the same argument gives
\(\sum_i z_i(r_i-z_i)\).  This proves \eqref{eq:internal}.
\end{proof}

\begin{example}
Take two rows with \(u=3\) and \(v=1\) ones in this component order
(Figure~\ref{fig:multitableau}).  The only inversion is
adjacent-column, so \(\inv=1\).  At level \(2\) the word of continuing
rows is \(10\), giving \(J=1\); at level \(3\) only the earlier row is
eligible.  Hence \(\inv+J=2\), once for each of the levels \(2,3\)
reached by exactly one of the two strings, as in
Lemma~\ref{lem:internal}.
\end{example}

\begin{figure}[H]
\centering
\begin{tikzpicture}[line join=round,
  cell/.style={draw, minimum size=0.72cm, anchor=south west,
    font=\small, inner sep=0pt},
  one/.style={cell, fill=onefill},
  two/.style={cell, fill=twofill}]
\def\cs{0.72}
% ---- LEFT: multitableau, r=(3,1), s=2, a=(2,1), z=(3,1) ----
% s-block: rows u=2, u=1
\node[one] at (0*\cs,5*\cs) {$1$}; \node[one] at (1*\cs,5*\cs) {$1$};
\node[one] at (0*\cs,4*\cs) {$1$}; \node[two] at (1*\cs,4*\cs) {$2$};
\node[left,font=\small] at (-0.25,4.75*\cs) {$s$-block};
% beta-block: row lengths 1,1,2
\node[one] at (0*\cs,2.2*\cs) {$1$};
\node[one] at (0*\cs,1.2*\cs) {$1$};
\node[one] at (0*\cs,0.2*\cs) {$1$}; \node[one] at (1*\cs,0.2*\cs) {$1$};
\node[left,font=\small] at (-0.25,1.4*\cs) {$\beta$-block};
\node[right,font=\small] at (2.1,4.7*\cs) {$a_1=2,\ a_2=1$};
\node[right,font=\small] at (2.1,1.1*\cs) {$z_1=3,\ z_2=1$};
% ---- RIGHT: inversion rule ----
\begin{scope}[xshift=7.6cm]
\node[font=\small] at (1.9*\cs,5.6*\cs) {inversion rule};
\node[font=\footnotesize] at (1.9*\cs,5.0*\cs) {(earlier $u=3$, later $v=1$)};
\node[one] at (0*\cs,3*\cs) {$1$};\node[one] at (1*\cs,3*\cs) {$1$};
  \node[one] at (2*\cs,3*\cs) {$1$};\node[two] at (3*\cs,3*\cs) {$2$};
\node[left,font=\footnotesize] at (-0.15,3.5*\cs) {earlier};
\node[one] at (0*\cs,1.6*\cs) {$1$};\node[two] at (1*\cs,1.6*\cs) {$2$};
  \node[two] at (2*\cs,1.6*\cs) {$2$};\node[two] at (3*\cs,1.6*\cs) {$2$};
\node[left,font=\footnotesize] at (-0.15,2.1*\cs) {later};
% arrow: earlier col. 3 ("1") down to later col. 2 ("2")
\draw[red,thick,->] (2.5*\cs,3.0*\cs) -- (1.5*\cs,2.6*\cs);
\node[right,font=\footnotesize,red] at (3*\cs+0.15,1.0*\cs) {$\inv=1$};
\end{scope}
\end{tikzpicture}
\caption{A multitableau for \(\ell=2\), \(\boldsymbol r=(3,1)\),
\(s=2\); \(1\)'s are blue, \(2\)'s tan.  Right: the sole inversion of
the pair shown is adjacent-column (red arrow).}
\label{fig:multitableau}
\end{figure}
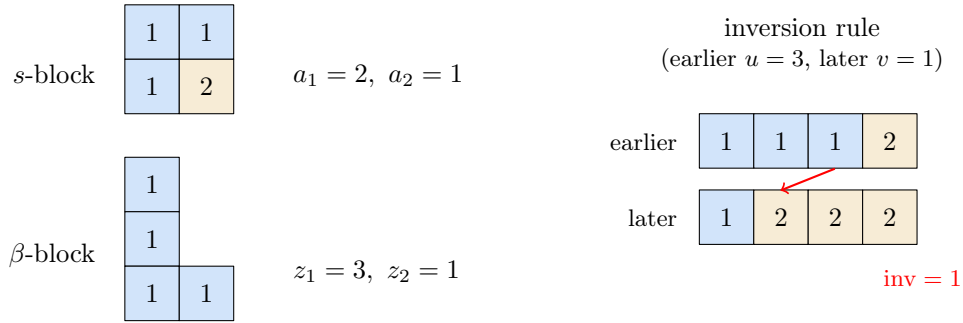

If the \(\beta\)-block precedes the \(s\)-block, the cross-block
inversions are
\begin{equation}\label{eq:cross-minus}
C^-:=\sum_{i=1}^{\ell}
\bigl((r_i-z_i)a_i+z_{i+1}(s-a_i)\bigr).
\end{equation}
Indeed, column \(i\) holds \(r_i-z_i\) entries \(2\) in the earlier
\(\beta\)-block and \(a_i\) entries \(1\) in the later block, while the
adjacent-column count pairs \(z_{i+1}\) earlier entries \(1\) in column
\(i+1\) with \(s-a_i\) later entries \(2\) in column \(i\); these are
the two displayed terms.  Reindexing the second terms gives
\[
C^-=\sum_{i=1}^{\ell}
\bigl(r_i a_i+(s-a_{i-1}-a_i)z_i\bigr).
\]
By Lemma~\ref{lem:internal}, \eqref{eq:Mlevels} and
\eqref{eq:cross-minus}, a term-by-term simplification gives
\begin{align}
sM-\inv(T)
&=J_a(T)+J_z(T)+\sum_{i=1}^{\ell}
\bigl(a_i^2+z_i^2+a_i z_i+z_i a_{i-1}
-r_i(a_i+z_i)\bigr)\notag\\
&=J_a(T)+J_z(T)+E^-(\boldsymbol a,\boldsymbol z).
\label{eq:inv-Eminus}
\end{align}

If the \(s\)-block precedes the \(\beta\)-block, the cross-block count
is instead
\begin{displaymath}
C^+:=\sum_{i=1}^{\ell}
\bigl((s-a_i)z_i+(r_i-z_i)a_{i+1}\bigr).
\end{displaymath}
The first term counts same-column pairs with an earlier \(2\) and a
later \(1\), the second adjacent-column pairs with an earlier \(1\) in
column \(i+1\) and a later \(2\) in column \(i\).  By
Lemma~\ref{lem:internal}, we get
\begin{equation}\label{eq:inv-Eplus}
sM-\inv(T)
=J_a(T)+J_z(T)+E^+(\boldsymbol a,\boldsymbol z).
\end{equation}

\begin{proposition}[Left-hand supernomial identity]
\label{prop:left-supernomial}
For either sign, we have
\begin{equation}\label{eq:left-supernomial}
\sum_{\substack{\boldsymbol a,\boldsymbol z\\
\sum_i(a_i+z_i)=M}}
q^{E^\pm(\boldsymbol a,\boldsymbol z)}
B_{\boldsymbol r,s}(\boldsymbol a,\boldsymbol z)
=q^{sM}\widetilde S_{(M,|\mu|-M),\mu}(q^{-1}).
\end{equation}
By content symmetry, the supernomial on the right may equally be
written with the two parts of its lower index interchanged, as
\(\widetilde S_{(|\mu|-M,\,M),\,\mu}(q^{-1})\).
Equivalently, in terms of the \defn{coinversion} statistic
\begin{equation}\label{eq:coinv-def}
\coinv(T):=n(\mu)-\inv(T),
\qquad
n(\mu)=\sum_{i=1}^{\ell}\binom{r_i+s}{2},
\end{equation}
we have
\begin{equation}\label{eq:left-coinv}
\sum_{\substack{\boldsymbol a,\boldsymbol z\\
\sum_i(a_i+z_i)=M}}
q^{E^\pm(\boldsymbol a,\boldsymbol z)}
B_{\boldsymbol r,s}(\boldsymbol a,\boldsymbol z)
=q^{\,sM-n(\mu)}\sum_T q^{\coinv(T)}
=q^{\,sM-n(\mu)}\,S_{(M,|\mu|-M),\mu}(q),
\end{equation}
the sum over \(T\) as in
\eqref{eq:tableau-supernomial}, those with \(M\) entries equal
to \(1\).  Here \(\coinv(T)\geq0\), and
\(S_{(M,|\mu|-M),\mu}(q)\) is the unnormalized supernomial
\eqref{eq:supernomial-def}.
\end{proposition}

Claims \eqref{eq:left-supernomial} and \eqref{eq:left-coinv}
for \emph{both} signs need order-independence, not content
symmetry: \(E^{\pm}\) arise from the two block orders, and the
sums agree because the multitableau generating function does not
see that order (Remark~\ref{rem:order-vs-content}).  Content
symmetry enters later to match Schilling's index convention.

\begin{remark}[Why \(n(\mu)\) and not \(sM\)]
\label{rem:normalizer}
The exponent \(sM-\inv(T)\) of
\eqref{eq:inv-Eminus}--\eqref{eq:inv-Eplus} can be negative, so it is
\(\coinv\) that carries the combinatorial content: at \(\ell=1\),
\(\boldsymbol r=(2)\), \(s=0\), \(M=1\), where \(s=0\) makes \(sM\)
vanish while \(\inv_z\) does not, the fixed-boundary sum is
\(1+q^{-1}\).  By contrast \(\inv(T)\leq n(\mu)\) always, since two
one-row components of lengths \(L,L'\) contribute at most
\(\min(L,L')\) inversions; summing by columns gives
\(n(\mu)=\sum_i\binom{\mu'_i}{2}\) as in \eqref{eq:coinv-def}, so
\(\coinv\geq0\).  The monomial \(q^{\,sM-n(\mu)}\) in
\eqref{eq:left-coinv} may still have negative exponent, and is removed
on reassembly (see Remark~\ref{rem:laurent}).
\end{remark}

\begin{proof}[Proof of Proposition~\ref{prop:left-supernomial}]
Combine \eqref{eq:Jgf}, the expansion
\eqref{eq:tableau-supernomial} (Proposition~\ref{prop:tableau-supernomial}),
and the exponent identities \eqref{eq:inv-Eminus} and
\eqref{eq:inv-Eplus}, which rest on
Lemma~\ref{lem:internal}.  The ordering with the \(\beta\)-block
before the adjoined \(s\)-block yields \(E^-\) via
\eqref{eq:inv-Eminus}; the reverse ordering yields \(E^+\) via
\eqref{eq:inv-Eplus}, and
Proposition~\ref{prop:tableau-supernomial} equates the two.  The content-symmetry restatement is
\eqref{eq:supernomial-tilde}.  For \eqref{eq:left-coinv}, the
definition \eqref{eq:supernomial-tilde} gives
\(\widetilde S_{\lambda\mu}(q^{-1})=q^{-n(\mu)}S_{\lambda\mu}(q)\),
so the right side of \eqref{eq:left-supernomial} is
\(q^{\,sM-n(\mu)}S_{\lambda\mu}(q)\); and by
\eqref{eq:tableau-supernomial},
\[
\sum_Tq^{\coinv(T)}
=q^{n(\mu)}\sum_Tq^{-\inv(T)}
=q^{n(\mu)}\widetilde S_{\lambda\mu}(q^{-1})
=S_{\lambda\mu}(q),
\]
with \(\lambda=(M,|\mu|-M)\).  Nonnegativity of the coefficients of
\(S_{\lambda\mu}(q)\) follows from
\eqref{eq:supernomial-def}, the Kostka numbers being nonnegative
integers and the Kostka--Foulkes polynomials having nonnegative
integer coefficients by the charge formula of
Lascoux--Sch\"utzenberger (see, e.g., \cite[Ch.~III]{Macdonald}).
\end{proof}

\subsection{The rigged-configuration expansion}
\label{subsec:rigged}

Apply Proposition~\ref{prop:SchillingA1} (equation
\eqref{eq:SchillingA1}) with \(c_i=r_i+s\), as in \eqref{eq:muprime}.  Then we have
$T_i=r_i-r_{i+1}+m_{i+1}$ for $i<\ell$, and
$T_\ell=r_\ell+s$. 
We also have
\[
\Phi(\boldsymbol m)
=\sum_{i=1}^{\ell-1}(r_{i+1}+s-m_{i+1})m_i.
\]
The elementary reciprocity formula
\begin{equation}\label{eq:qreciprocity}
\qbinomQ{N}{j}{q^{-1}}
=q^{-j(N-j)}\qbinom Nj
\end{equation}
now gives
\begin{align*}
q^{sM}\widetilde S_{(|\mu|-M,M),\mu}(q^{-1})
&=\sum_{\substack{\boldsymbol m\geq0\\\sum_i m_i=M}}
q^{sM-\Phi(\boldsymbol m)
-\sum_i m_i(T_i-m_i)}
\prod_{i=1}^{\ell}\qbinom{T_i}{m_i}.
\end{align*}
For \(i<\ell\), we have
\[
(r_{i+1}+s-m_{i+1})m_i
+m_i(T_i-m_i)=m_i(r_i+s-m_i),
\]
and the same equality for \(i=\ell\) is simply
\(m_\ell(T_\ell-m_\ell)=m_\ell(r_\ell+s-m_\ell)\).
Since \(M=\sum_i m_i\), the exponent reduces exactly to
\[
sM-\sum_{i=1}^{\ell}m_i(r_i+s-m_i)
=\sum_{i=1}^{\ell}(m_i^2-r_i m_i).
\]
\begin{proof}[Proof of Theorem~\ref{thm:master}]
Proposition~\ref{prop:left-supernomial} and the computation just made
give the fixed-\(M\) assertion; summing over \(M\) gives
\eqref{eq:master}.
\end{proof}

\subsection{Proof of the fixed-boundary identities}
\label{subsec:proof-fixed}

The minus and plus cases amount to the two orders of the same tableau
blocks.  We now combine the transformation with
Lemma~\ref{lem:Q-reindex}.

\begin{proof}[Proof of Theorem~\ref{thm:fixed-boundary}]
For \(b=3k-1\), the first reindexing of Lemma~\ref{lem:Q-reindex} gives
\(\ell=k-1\), \(s=r_k\) and Gaussian product \eqref{eq:Baz}, so by
\eqref{eq:Qminus} the left side of \eqref{eq:sum-minus} equals
\[
q^{\sum_{i=1}^{k}r_i^2}
\sum_{\boldsymbol a,\boldsymbol z}
q^{E^-(\boldsymbol a,\boldsymbol z)}
B_{\boldsymbol r,r_k}(\boldsymbol a,\boldsymbol z).
\]
Here Theorem~\ref{thm:master} has final Gaussian polynomial
\(\qbinom{r_{k-1}+r_k}{m_{k-1}}\), and multiplying by the displayed
constant \(q\)-power gives \eqref{eq:sum-minus}.  For \(b=3k+1\), the
second reindexing of Lemma~\ref{lem:Q-reindex} gives \(\ell=k\),
\(s=r_k\) and \eqref{eq:Qplus}. Theorem~\ref{thm:master} with \(E^+\)
has final Gaussian polynomial
\(\qbinom{r_k+s}{m_k}=\qbinom{2r_k}{m_k}\), and multiplying by
\(q^{\sum_i r_i^2}\) gives \eqref{eq:sum-plus}.
\end{proof}

\subsection{The fixed-boundary identities to the product}
\label{subsec:product}

It remains to restore the boundary summation.  The next lemma checks,
factor by factor, that the Gaussian weights above reconstruct the
Pochhammer quotient defining \(Z_{3,b}\).  Write
\(S_b(\boldsymbol r)\) for the left side of \eqref{eq:sum-plus} or
\eqref{eq:sum-minus}.

\begin{lemma}[The HJO series]\label{lem:reassemble}
For \(b=3k\pm1\), we have
\begin{equation}\label{eq:reassemble}
Z_{3,b}(q)=
\sum_{r_1\geq\cdots\geq r_k\geq0}
\frac{1}{(q)_{r_1}}
\left(\prod_{i=1}^{k-1}\qbinom{r_i}{r_{i+1}}\right)
S_b(\boldsymbol r).
\end{equation}
\end{lemma}

\begin{proof}
Partition the sum \eqref{eq:inf-sum-def} according to the values
\eqref{eq:fixed-r}.  The nonzero Gaussian conditions in
\eqref{eq:domain}, together with
\(r_1\geq\cdots\geq r_k\geq0\), are exactly the inequalities
defining \(\mathcal C_{3,b}\) in these coordinates.  We verify the
Pochhammer factors explicitly.
Write \(\Pi_{3,b}(\boldsymbol n)\) for the Pochhammer product in the
summand of \(Z_{3,b}\) (not to be confused with the infinite product
\(P_{3,b}(q)\) of \eqref{eq:Pab}), so that
\[
\Pi_{3,b}(\boldsymbol n):=
\prod_{g\in\Gap_{3,b}}
\frac{(q)_{n_g-n_{g-3-b}}}
{(q)_{n_g-n_{g-3}}(q)_{n_g-n_{g-b}}}.
\]
For \(b=3k-1\), use the notation of
Lemma~\ref{lem:Q-reindex}, put \(\ell=k-1\), and set
\(a_{\ell+1}=z_{\ell+1}=0\).  The predecessors of each type of gap
variable are
\[
\begin{array}{|c|ccc|}
\hline
n_g&n_{g-3}&n_{g-b}&n_{g-b-3}\\ \hline
a_i&a_{i+1}&0&0\\
z_i&z_{i+1}&0&0\\
r_i\ (i\leq\ell)&r_{i+1}&z_i&z_{i+1}\\
r_k=s&a_1&0&0 .\\ \hline
\end{array}
\]
Substitution into the definition of \(\Pi_{3,b}\) gives
\begin{align}
\Pi_{3,3k-1}(\boldsymbol n)=
\prod_{i=1}^{\ell}
\frac{(q)_{r_i-z_{i+1}}}
{(q)_{r_i-r_{i+1}}(q)_{r_i-z_i}}
\cdot\frac{1}{(q)_{s-a_1}}
\cdot
\prod_{i=1}^{\ell}
\frac{1}
{(q)_{a_i-a_{i+1}}(q)_{z_i-z_{i+1}}}.
\label{eq:Poch-minus}
\end{align}

For \(b=3k+1\), put \(a_{k+1}=z_{k+1}=0\).  The corresponding
table is
\[
\begin{array}{|c|ccc|}
\hline
n_g&n_{g-3}&n_{g-b}&n_{g-b-3}\\ \hline
a_i&a_{i+1}&0&0\\
z_i&z_{i+1}&0&0\\
r_i\ (i<k)&r_{i+1}&z_i&z_{i+1}\\
r_k=s&a_1&z_k&0 .\\ \hline
\end{array}
\]
It follows that
\begin{align}
\Pi_{3,3k+1}(\boldsymbol n)
=
\prod_{i=1}^{k-1}
\frac{(q)_{r_i-z_{i+1}}}
{(q)_{r_i-r_{i+1}}(q)_{r_i-z_i}}
\cdot
\frac{(q)_{r_k}}
{(q)_{r_k-a_1}(q)_{r_k-z_k}}
\cdot
\prod_{i=1}^{k}
\frac{1}
{(q)_{a_i-a_{i+1}}(q)_{z_i-z_{i+1}}}.
\label{eq:Poch-plus}
\end{align}

It remains to compare these formulas with the Gaussian products.
Expanding \eqref{eq:Baz} gives
\begin{align*}
B_{\boldsymbol r,s}(\boldsymbol a,\boldsymbol z)
=\prod_{i=1}^{\ell}
\frac{(q)_{a_{i-1}}}
{(q)_{a_i}(q)_{a_{i-1}-a_i}}\cdot
\prod_{i=1}^{\ell}
\frac{(q)_{r_i-z_{i+1}}}
{(q)_{z_i-z_{i+1}}(q)_{r_i-z_i}}.
\end{align*}
The fixed-chain factor is
\[
\frac{1}{(q)_{r_1}}
\prod_{i=1}^{k-1}\qbinom{r_i}{r_{i+1}}
=\frac{1}
{(q)_{r_k}\prod_{i=1}^{k-1}(q)_{r_i-r_{i+1}}},
\]
where the right-hand side follows because the numerators
\((q)_{r_i}\) of the Gaussian polynomials telescope against the
denominators \((q)_{r_{i+1}}\), leaving only \(1/(q)_{r_1}\) times
\((q)_{r_1}\) at the top and the terminal \(1/(q)_{r_k}\).  Multiply
this fixed-chain factor by \(B_{\boldsymbol r,s}\).  The
\(\boldsymbol a\)-part of \(B_{\boldsymbol r,s}\) telescopes in the
same way. Namely, with \(a_0=s=r_k\), we have
\[
\prod_{i=1}^{\ell}
\frac{(q)_{a_{i-1}}}{(q)_{a_i}(q)_{a_{i-1}-a_i}}
=\frac{(q)_{a_0}}{(q)_{a_\ell}\prod_{i=1}^{\ell}(q)_{a_{i-1}-a_i}}
=\frac{(q)_{r_k}}{(q)_{a_\ell}\prod_{i=1}^{\ell}(q)_{a_{i-1}-a_i}},
\]
since each interior \((q)_{a_i}\) appears once in a numerator and
once in a denominator.  For instance, the denominator \((q)_{a_1}\) in the
\(i=1\) term is cancelled by the numerator \((q)_{a_1}\) in the \(i=2\)
term, and so on down the chain.  In the
minus case, we have \(\ell=k-1\) and \(a_\ell=a_{k-1}\). The surviving
numerator \((q)_{a_0}=(q)_{r_k}\) cancels the terminal
\(1/(q)_{r_k}\) of the fixed-chain factor, and what remains is
exactly \eqref{eq:Poch-minus}.  In the plus case \(\ell=k\), the
boundary predecessor of \(r_k=s\) additionally contributes
\(z_k\) (see the last row of the plus table), so the terminal
numerator \((q)_{r_k}\) is displayed rather than cancelled, giving
\eqref{eq:Poch-plus}.  Therefore, we have
\[
\Pi_{3,b}(\boldsymbol n)
=\frac{1}{(q)_{r_1}}
\left(\prod_{i=1}^{k-1}\qbinom{r_i}{r_{i+1}}\right)
B_{\boldsymbol r,s}(\boldsymbol a,\boldsymbol z).
\]
The exponent is \(Q_{3,b}(\boldsymbol n)\) in both expressions.
Summing first over the free variables and then over the fixed
variables proves \eqref{eq:reassemble}.
\end{proof}

With the reassembly in hand, we can dispose of the tableau model
announced in the introduction.

\begin{proof}[Proof of Corollary~\ref{cor:tableau-model}]
By Lemma~\ref{lem:Q-reindex} and the proof of
Theorem~\ref{thm:fixed-boundary}, in both congruence classes
\[
S_b(\boldsymbol r)=
q^{\sum_{i=1}^{k}r_i^2}
\sum_{\boldsymbol a,\boldsymbol z}
q^{E^{\pm}(\boldsymbol a,\boldsymbol z)}
B_{\boldsymbol r,s}(\boldsymbol a,\boldsymbol z),
\]
with \((\ell,s)\) as in the corollary.  Splitting the inner sum
according to \(M=\sum_i(a_i+z_i)\) and applying
\eqref{eq:left-coinv} of Proposition~\ref{prop:left-supernomial}
gives the two equivalent forms of that identity, with
\(\mu'=(r_1+s,\ldots,r_\ell+s)\) as in \eqref{eq:muprime}.
Substituting the coinversion form into \eqref{eq:reassemble} of
Lemma~\ref{lem:reassemble} gives \(Z_{3,b}(q)\).
\end{proof}

The product evaluation uses the theta functions \eqref{eq:theta-def}
of Section~\ref{sec:schilling-warnaar}.

\begin{proof}[Proof of Theorem~\ref{thm:main}]
We prove (2) first.
 Insert Theorem~\ref{thm:fixed-boundary}
into \eqref{eq:reassemble}.  If \(b=3k-1\), the result is
\begin{align*}
\sum_{\substack{r_1,\ldots,r_k\geq0\\
m_1,\ldots,m_{k-1}\geq0}}
\frac{q^{\,r_k^2+
\sum_{i=1}^{k-1}(r_i^2-r_i m_i+m_i^2)}}{(q)_{r_1}}
\prod_{i=1}^{k-1}\qbinom{r_i}{r_{i+1}}
\cdot
\qbinom{r_{k-1}+r_k}{m_{k-1}}
\prod_{i=1}^{k-2}
\qbinom{r_i-r_{i+1}+m_{i+1}}{m_i}.
\end{align*}
This is exactly the left-hand side of Warnaar's identity
\eqref{eq:WarnaarMinus}.  Hence the sum equals
\begin{equation}\label{eq:prod-minus}
\frac{(q^{3k+2};q^{3k+2})_\infty^2}{(q;q)_\infty^2}
\theta(q^k,q^{k+1},q^{k+1};q^{3k+2}).
\end{equation}

If \(b=3k+1\), substitution gives
\begin{align*}
\sum_{\substack{r_1,\ldots,r_k\geq0\\
m_1,\ldots,m_k\geq0}}
\frac{q^{\sum_{i=1}^{k}(r_i^2-r_i m_i+m_i^2)}}
{(q)_{r_1}}
\left(\prod_{i=1}^{k-1}\qbinom{r_i}{r_{i+1}}\right)
\qbinom{2r_k}{m_k}\cdot
\prod_{i=1}^{k-1}
\qbinom{r_i-r_{i+1}+m_{i+1}}{m_i}.
\end{align*}
This is the left-hand side of Warnaar's identity
\eqref{eq:WarnaarPlus} (with \(n_i=r_i\)). The number and range of
both families of variables agree, and \(k\geq1\) holds, including at
the endpoint \(b=4\).  Hence it equals
\begin{equation}\label{eq:prod-plus}
\frac{(q^{3k+4};q^{3k+4})_\infty^2}{(q;q)_\infty^2}\cdot
\theta(q^{k+1},q^{k+1},q^{k+2};q^{3k+4}).
\end{equation}

Now we identify the products.  Put \(L=b+3\) and
specialize \eqref{eq:charge} to \(a=3\).  Since \(b>3\), we
have \(\min\{a,b\}=3\), so we only have the truncation by \(a\), and
we find that
\[
r_{3,b}(i)=
\begin{cases}
0,&3i\equiv\pm1\pmod L,\\
1,&3i\equiv\pm2\pmod L,\\
2,&\text{otherwise},
\end{cases}
\qquad r_{3,b}(0)=0.
\]
If \(b=3k-1\), so \(L=3k+2\), then we have
$3(k+1)\equiv1\pmod L$ and $3k\equiv-2\pmod L$.
Thus, the charge is \(0\) on
\(0,\pm(k+1)\), is \(1\) on \(\pm k\), and is \(2\) elsewhere.
If \(b=3k+1\), so \(L=3k+4\), then
$3(k+1)\equiv-1\pmod L$ and $3(k+2)\equiv2\pmod L$,
so the charge is \(0\) on \(0,\pm(k+1)\), is \(1\) on
\(\pm(k+2)\), and otherwise is \(2\).

Now \(1/(q;q)_\infty^2\) supplies exponent \(2\) in every residue class,
the numerator \((q^L;q^L)_\infty^2\) cancels it in the zero class, and
each \(\theta(q^d;q^L)\) lowers by one the exponents in the two distinct
classes \(\pm d\).  No \(d\) here is self-paired: \(d\in\{k,k+1\}\) with
\(L=3k+2\) in the minus case and \(d\in\{k+1,k+2\}\) with \(L=3k+4\) in
the plus case, and neither equals \(L/2\).  Hence \(P_{3,3k-1}(q)\) is
\eqref{eq:prod-minus} and \(P_{3,3k+1}(q)\) is \eqref{eq:prod-plus}.  By
\eqref{eq:inf-sum-def} and \eqref{eq:Pab}, the two sides of
the display in part (2) are \(Z_{3,b}(q)\) and \(P_{3,b}(q)\), proving
part (2).

Part (1) is now part (2) together with
\(P_{3,b}^{\mathrm{pt}}(q)=Z_{3,b}(q)\), which is \eqref{eq:Ppt} with
\(a=3\).  By \eqref{eq:prod-minus} and \eqref{eq:prod-plus} the common
value equals \((q^L;q^L)_\infty^2/(q;q)_\infty^2\) times three theta
functions \(\theta(q^d;q^L)\) of \eqref{eq:theta-def} with \(L=b+3\), a
finite product of Jacobi theta quotients of level \(L\); each factor is,
up to a rational power of \(q\), a modular unit with character of level \(L\) in the
sense of \cite[Chapters~2--3]{KubertLang}.  This is part (3).
\end{proof}

\end{document}